\documentclass[a4paper,11pt]{article}
\usepackage[pdftex]{graphicx}
\usepackage{graphicx,a4wide,listings}
\usepackage[intlimits]{amsmath}
\usepackage{amsxtra, amssymb,fancyhdr, amsthm,latexsym}
\usepackage[english]{babel}
\usepackage[ansinew]{inputenc}
\usepackage[bookmarks]{}
\usepackage{enumerate}
\numberwithin{figure}{section}
 \usepackage[colorlinks=true]{hyperref}
\hypersetup{urlcolor=blue, citecolor=red}
\usepackage{enumitem}
\usepackage{tikz}
\usepackage{caption}
\usepackage{subcaption}

\begin{document}

\newtheorem{theorem}{Theorem}[section]
\newtheorem{corollary}[theorem]{Corollary}
\newtheorem{lemma}[theorem]{Lemma}
\newtheorem{proposition}[theorem]{Proposition}
\newtheorem{conjecture}[theorem]{Conjecture}
\newtheorem*{problem}{Problem}
\theoremstyle{definition}
\newtheorem{definition}[theorem]{Definition}
\theoremstyle{definition}
\newtheorem{remark}[theorem]{Remark}
\newcommand{\ep}{\varepsilon}
\newcommand{\eps}[1]{{#1}_{\varepsilon}}

\def\R{\mathbb{R}}
\def\Rp{\mathbb{R}^+}
\def\N{\mathbb{N}}
\def\Np{\mathbb{N}^+}
\def\Obj{\mathcal{O}}
\def\dOO{\partial\Omega_\mathcal{O}}
\def\OO{\Omega_{\mathcal{O}}}
\newcommand{\Int}[4]{\int_{#1}^{#2}\! #3 \, #4}
\newcommand{\pair}[2]{\left\langle #1 , #2 \right\rangle}
\newcommand{\map}[3]{#1 : #2 \rightarrow #3}

\title{Modelling with measures: Approximation of a mass-emitting object by a point source}
\author{Joep H.M.~Evers\thanks{Corresponding author; email: \texttt{j.h.m.evers@tue.nl}} \thanks{Institute for Complex Molecular Systems \& Centre for Analysis, Scientific computing and Applications, Eindhoven University of Technology, P.O.~Box 513, 5600 MB Eindhoven, The Netherlands. JE is financially supported by the Netherlands Organisation for Scientific Research (NWO), Graduate Programme 2010.} \  \  Sander C.~Hille\thanks{Mathematical Institute, Leiden University, P.O.~Box 9512, 2300 RA Leiden, The Netherlands}  \  \ Adrian Muntean\thanks{Institute for Complex Molecular Systems \& Centre for Analysis, Scientific computing and Applications, Eindhoven University of Technology, P.O.~Box 513, 5600 MB Eindhoven, The Netherlands}}
\date{\today}
\maketitle

\begin{abstract}
We consider a linear diffusion equation on $\Omega:=\R^2\setminus\overline{\OO}$, where $\OO$ is a bounded domain. The time-dependent flux on the boundary $\Gamma:=\dOO$ is prescribed. The aim of the paper is to approximate the dynamics by the solution of the diffusion equation on the whole of $\R^2$ with a measure-valued point source in the origin and provide estimates for the quality of approximation. For all time $t$, we derive an $L^2([0,t];L^2(\Gamma))$-bound on the difference in flux on the boundary. Moreover, we derive for all $t>0$ an $L^2(\Omega)$-bound and an $L^2([0,t];H^1(\Omega))$-bound for the difference of the solutions to the two models.
\end{abstract}

\section{Introduction}
``What is the force on a test charge due to a single point charge $q$ which is at rest a distance $r$ away?" is a common type of question in textbooks about electromagnetism (e.g.~\cite{Griffiths}, p.~59). In reality there is of course no such thing as a point charge having no volume. This is just a simplification due to the fact that the volume of the charged particle is very small compared to the other typical length scales in the system. Throughout physics it is common practice to replace objects of negligible size by point masses. For instance, grains or colloids in a solution \cite{Jaeder-Nagel}, crowd dynamics \cite{Gulikers_Jstat}, electrostatics \cite{Jackson}, defects in crystalline structures \cite{LeBris, vMeurs}. Of particular interest is the setting in which the exchange of mass, energy etc.~between the interior and the exterior of the object takes place at its boundary. In this case the object is approximated not by a mere point mass, but by a point source. Experimental evidence suggests that this example of `modelling with measures' is often a good approximation to the original (spatially extended) system. In this paper, we consider the problem of quantifying the accuracy of this type of approximation, focussing on a simple scenario.\\
\\
In $\R^2$, we consider an object of fixed shape and position and of finite size. Outside the object there is a concentration of mass that evolves by diffusion. On the boundary of the object there is prescribed mass flux in normal direction. This flux is a simplistic way of describing the result of processes that occur in the interior of the object. We wish to approximate this object by a point source. To this aim we replace the original diffusion equation on the exterior domain $\Omega$ by a diffusion equation on the whole of $\R^2$ with a Dirac measure included at its right-hand side. The exact formulation of the equations will be made clear in Section \ref{sec:two problems}.\\
\\
This is a first step towards modelling and analysing the mass distribution dynamics in realistic settings involving a large number of small objects  moving around in a bounded domain while exchanging mass. Our motivation comes from the intracellular transport of chemical compounds in vesicles, like neurotransmitters in neurons (cf.~\cite{Lin-Edwards:1997}) or the hypothetical vesicular transport mechanism for the plant hormone auxin proposed in \cite{Baluska} as an alternative to the conventional auxin transport paradigm (in analogy to neurotransmitters). Auxin is a crucial molecule regulating growth and shape in plants. The vesicles are small membrane-bound balls covered by specific transmembrane transporter proteins that take up auxin from the surrounding cytoplasm. The vesicles are driven by molecular motors over a network of intracellular filaments \cite{Hirokawa_ea,Raven}, e.g. from one end of the cell to the other as in Polar Auxin Transport (PAT). Experimental investigations of PAT in {\em Chara} species \cite{Boot} revealed that neither diffusion nor cytoplasmic streaming can be the driving mechanism of PAT in the long (3-8 cm) internodal {\em Chara} cells. See \cite{Boot,Raven} for further discussion and an overview.\\
\\
A substantial amount of mathematical modelling efforts on PAT have focussed on pattern formation in plant cell tissues (see \cite{vBerkel,Kramer,Merks} and the references cited therein). Upscaling to an effective macroscopic continuum description for transport at tissue level was considered in \cite{Chavarria}. All models are based however on the assumption of diffusion as intracellular transport mechanism for auxin. Ultimately, we aim at obtaining a convenient mathematical description of the vesicle-driven transport dynamics \textit{within} a cell, in particular in terms of an effective continuum model, which is needed to replace diffusion in an upscaling argument similar to \cite{Chavarria}. In view of (the absence of) relevant mathematical literature, this perspective seems to be rather unexplored.\\
\\
Why do we insist on introducing measures to this problem? This modelling strategy is especially useful once we wish to describe the interaction between multiple moving objects (vesicles). We expect the mathematical description to be much simpler in terms of discrete measures (i.e.~the weighted sum of Dirac measures) and the analysis and numerical approximation likewise (see, for instance,  \cite{Rude,Seidman} for a related case). But before we can go to this advanced setting, we first need to investigate the quality of the approximation for a simple reference scenario; this is the main concern of this paper.\\
\\
After the aforementioned overview of model equations in Section \ref{sec:two problems}, we summarize in Section \ref{sec:mainresults} the main (boundedness) results of this paper, followed by some useful preliminaries in Section \ref{sec:preliminaries}. In Section \ref{sec:flux estimates} we show boundedness of the difference in the flux of the full problem (including the finite-size object) and the flux of the reduced problem (including the point source). This result is used in Section \ref{sec:exterior estimates}, where we estimate the difference between the two problems' solutions on the exterior domain.

\section{Two problems}\label{sec:two problems}
Let $\OO\in\R^2$ be an open and bounded domain, such that its boundary $\Gamma:=\dOO$ is $C^2$ and has finite length. This set denotes the interior of an object $\Obj$ with mass-exchange at its boundary. We assume $0\in\OO$. Let $\Omega$ denote the exterior of $\Obj$. That is, $\Omega:=\R^2\setminus\overline{\OO}$. See Figure \ref{fig:domain Omega} for a sketch of the geometry.\\
\\
For given initial condition $\map{u_0}{\Omega}{\Rp}$ and given flux $\map{\phi}{\Gamma\times[0,T]}{\R}$, we consider the problem
\begin{align}\label{eqn:system Full problem}
\left\{
  \begin{array}{ll}
    \dfrac{\partial u}{\partial t}= d \Delta u, & \hbox{on $\Omega\times\Rp$;} \\
    u(0) = u_0, & \hbox{on $\Omega$;} \\
    d\nabla u\cdot n = \phi, & \hbox{on $\Gamma\times\Rp$.}
  \end{array}
\right.
\end{align}
Here, $d>0$ denotes the diffusion coefficient, which is fixed throughout this paper. The vector $n$ denotes the unit normal pointing outwards on $\Gamma$ (so \textit{into} $\OO$), and $\phi$ is the \textit{influx} of $u$ w.r.t.~$\Omega$. Positive $\phi$ corresponds to flux in the direction of $-n$.\\
\\
Use $\map{v_0}{\overline{\OO}}{\Rp}$ to define $\map{\hat{u}_0}{\R^2}{\Rp}$, given by
\begin{align}\label{eqn:extended initial condition}
\hat{u}_0 := \left\{
  \begin{array}{ll}
    u_0, & \hbox{on $\Omega$;} \\
    v_0, & \hbox{on $\overline{\OO}$,}
  \end{array}
\right.
\end{align}
which is an extension of $u_0$ to the whole of $\R^2$. The aim of the paper is to quantify the quality of approximation of the solution of \eqref{eqn:system Full problem} (with an appropriate solution concept, see Section \ref{sec:soln reg} below) with the restriction to $\Omega$ of the \textit{mild solution} of the problem
\begin{align}\label{eqn:system point source uhat}
\left\{
  \begin{array}{ll}
    \dfrac{\partial \hat{u}}{\partial t}= d \Delta \hat{u} + \bar{\phi}\delta_0, & \hbox{on $\R^2\times\Rp$;} \\
    \hat{u}(0) = \hat{u}_0, & \hbox{on $\R^2$,}
  \end{array}
\right.
\end{align}
(see also Section \ref{sec:soln reg}).
\begin{remark}
Typically, $\Obj$ is \textit{small} (we are deliberately vague in what sense), but even if that is not the case, the approach of this paper gives information about how much the solutions of the two problems deviate on $\Omega$. It is not our objective to investigate the behaviour of \eqref{eqn:system Full problem} in the limit $|\Obj|\to 0$. $\Obj$ keeps physical proportions.
\end{remark}
\begin{remark}
In \eqref{eqn:system point source uhat}, we have introduced a mapping $\map{\bar{\phi}}{\Rp}{\R}$ which represents the magnitude of the mass source. A measure-valued source was treated, for instance, in \cite{Seidman} (in the context of numerical approximation schemes) or in \cite{Boccardo_ea}; see also \cite{Lions} for more background on the solvability of such evolution equations.
\end{remark}
\begin{remark}
Problem \eqref{eqn:system point source uhat} is posed on the whole of $\R^2$. The boundary $\Gamma$ has no physical meaning in this problem; see Figure \ref{fig:domain point source}. However, the flux on this imaginary curve will be used in later estimates.
\end{remark}
\begin{figure}
        \centering
        \begin{subfigure}[b]{0.45\textwidth}
                \begin{tikzpicture}[scale=0.5, >= latex]
                \shade[shading=radial, even odd rule, inner color=gray!40!white]
                (-5,-5) rectangle (5,5)
                (0,2.5) to [out=0, in=135] (2,1.5) to [out=315, in=90] (3,0.5) to [out=270, in=45] (1.5,-1.5) to [out=225, in=0] (0,-2.5) to [out=180, in=315] (-1,-2) to [out=135, in=225] (-1,1.5) to [out=45, in=180] (0,2.5);
                \fill[gray!40!white] (0,2.5) to [out=0, in=135] (2,1.5) to [out=315, in=90] (3,0.5) to [out=270, in=45] (1.5,-1.5) to [out=225, in=0] (0,-2.5) to [out=180, in=315] (-1,-2) to [out=135, in=225] (-1,1.5) to [out=45, in=180] (0,2.5);
                \draw[->, color=gray!60!white] (0,-5.5)--(0,5.5);
                \draw[->, color=gray!60!white] (-5.5,0)--(5.5,0);
                \draw[line width = 1.5 pt] (0,2.5) to [out=0, in=135] (2,1.5) to [out=315, in=90] (3,0.5) to [out=270, in=45] (1.5,-1.5) to [out=225, in=0] (0,-2.5) to [out=180, in=315] (-1,-2) to [out=135, in=225] (-1,1.5) to [out=45, in=180] (0,2.5);
                \node at (1,-0.5){$\OO$};
                \node at (3,1.5){$u_0$};
                \node at(-2.5,-1){$\Omega$};
                \draw[->] (-1,-2)--(-0.5,-1.5) node[anchor=south west]{$n$};
                \draw[->] (1.5,-1.5)--(2.5,-2.5) node[anchor=south west]{$\phi$};
                \draw (-1,1.5) to [out=135, in=270] (-1.25,2) node[anchor=south east]{$\Gamma$};
                \end{tikzpicture}
                \caption{Original domain}
                \label{fig:domain Omega}
        \end{subfigure}%
        ~ 
        \begin{subfigure}[b]{0.45\textwidth}
                \begin{tikzpicture}[scale=0.5, >= latex]
                \shade[shading=radial, inner color=gray!40!white]
                (-5,-5) rectangle (5,5);
                \draw[->, color=gray!60!white] (0,-5.5)--(0,5.5);
                \draw[->, color=gray!60!white] (-5.5,0)--(5.5,0);
                \draw[loosely dotted] (0,2.5) to [out=0, in=135] (2,1.5) to [out=315, in=90] (3,0.5) to [out=270, in=45] (1.5,-1.5) to [out=225, in=0] (0,-2.5) to [out=180, in=315] (-1,-2) to [out=135, in=225] (-1,1.5) to [out=45, in=180] (0,2.5);
                \draw[->] (0,0)--(1,1);
                \draw[->] (0,0)--(1,-1);
                \draw[->] (0,0)--(-1,1);
                \draw[->] (0,0)--(-1,-1);
                \draw (0,0) node[fill, circle, scale=0.7]{};
                \draw (0,0) node[anchor=west]{$\hspace{5pt}\bar{\phi}\delta_0$};
                \node at (3,1.5){$u_0$};
                \node at (2.2,0.7){$v_0$};
                \end{tikzpicture}
                \caption{Extended domain}
                \label{fig:domain point source}
        \end{subfigure}
        \caption{(A): Typical example of the original domain $\Omega$ outside the object $\Obj$, on which $u$ evolves according to \eqref{eqn:system Full problem} starting from initial condition $u_0$. Also, $\phi$ and $n$, related to the boundary condition on $\Gamma$, are indicated. (B): Domain for the reduced problem associated to (A). $\Gamma$ is now an imaginary curve within the domain (to be used later). The initial conditions $u_0$ and $v_0$ hold outside and inside $\Gamma$, respectively. The point source of magnitude $\bar{\phi}$ is indicated in the origin.}\label{fig:domains}
\end{figure}
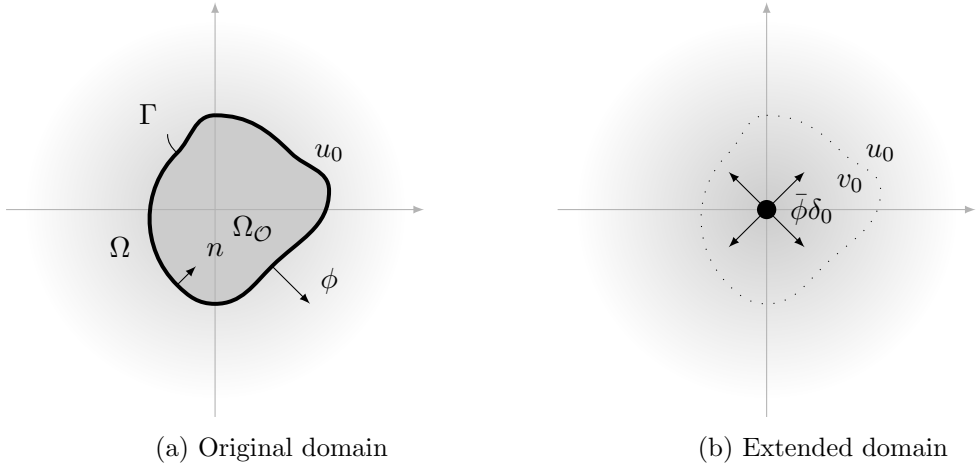

\section{Summary of the main results}\label{sec:mainresults}
In Section \ref{sec:soln reg} we shall use available results on maximal regularity that establish the existence of a unique solution $u$ to Problem \eqref{eqn:system Full problem} in the sense of $L^2(\Omega)$-valued distributions, provided the initial condition $u_0\in H^1(\Omega)$ and the prescribed flux $\phi\in H^1([0,T], L^2(\Gamma))\cap L^2([0,T],H^1(\Gamma))$. Mild solutions to Problem \eqref{eqn:system point source uhat} exist in a suitable Banach space containing the finite Borel measures for any initial {\em measure}, provided $\bar{\phi}\in L^1_{\mathrm{loc}}(\Rp)$ (see Section \ref{sec:soln reg}). We show that for more regular initial condition $\hat{u}_0\in H^1(\R^2)$ and flux from the source $\bar{\phi}\in H^1([0,T])$, the restriction of the mild solution $\hat{u}$ to $\Omega$ is as regular as $u$ on $\Omega$ (Theorem \ref{thm:reg}), namely
\[
u, \hat{u} \in H^1([0,T], L^2(\Omega))\cap L^2([0,T],H^2(\Omega)).
\]
Consequently, the time-averaged deviation between the prescribed flux $\phi$ on $\Gamma$ in Problem \eqref{eqn:system Full problem} and the flux on $\Gamma$ generated by the solution to Problem \eqref{eqn:system point source uhat} with flux $\bar{\phi}$ at $0$, i.e.
\begin{equation}\label{eq:def c-star}
c^*(t):=\Int{0}{t}{\|\phi(\tau)-d\nabla\hat{u}(\tau)\cdot n\|^2_{L^2(\Gamma)}}{d\tau}
\end{equation}
is finite for all $t\geq 0$. In Section \ref{sec:flux estimates} we derive an upper bound on $c^*(t)$, see Theorem \ref{thm:claim boundary} in terms of the data for Problems \eqref{eqn:system Full problem} and \eqref{eqn:system point source uhat}.
\vskip 0.2cm

\noindent Our main result is the following:
\begin{theorem}\label{thm:estimate u exterior}
Let $T>0$ and let the data for Problems \eqref{eqn:system Full problem} and \eqref{eqn:system point source uhat} satisfy $u_0\in H^1(\Omega)$, $\phi\in H^1([0,T], L^2(\Gamma))\cap L^2([0,T],H^1(\Gamma))$, $\bar{\phi}\in H^1([0,T])$ and $\hat{u}_0\in H^1(\R^2)$ is such that $\nabla\hat{u}_0\in L^p(\R^2)$ for some $2<p<\infty$. Then the unique solutions $u$ and $\hat{u}$ to \eqref{eqn:system Full problem} and \eqref{eqn:system point source uhat} are such, that for all $\ep\in(0,2d)$ there are $c_1, c_2>0$ such that
\begin{align}
\|u(\cdot,t)-\hat{u}(\cdot,t)\|^2_{L^2(\Omega)} &\leq c_1\,c^*(t)\,\textrm{e}^{\ep t},\text{ and}\label{eqn:thm bound L2}\\
\Int{0}{t}{\|u-\hat{u}\|^2_{H^1(\Omega)}}{} &\leq c_2\,c^*(t)\,\textrm{e}^{\ep t}.\label{eqn:thm bound H1 integrated}
\end{align}
for all $0<t\leq T$. The constants depend on $\Omega$, $d$ and $\ep$.
\end{theorem}
\begin{remark}
Note that the initial condition $\hat{u}_0$ needs to be more regular than `just' $H^1(\R^2)$ as needed in the regularity result for $\hat{u}$. The flux estimates in Section \ref{sec:flux estimates} require $\nabla\hat{u}_0\in L^p(\R^2)$ with $2<p<\infty$. The Sobolev Embedding Theorem (cf. \cite{AdamsFournier}, Thrm. 4.12, p. 85) yields that $\hat{u}_0\in H^2(\R^2)$ is a sufficient condition to have the stronger result that $\hat{u}_0\in H^1(\R^2)\cap W^{1,p}(\R^2)$ for any $2<p<\infty$. In that case necessarily $u_0\in H^2(\Omega)$ too.
\end{remark}

An important characteristic of estimates \eqref{eqn:thm bound L2} and \eqref{eqn:thm bound H1 integrated} is that the upper bounds are linear in $c^*(t)$. This implies that, if we manage to enforce $c^*(t)$ to be small, then also the solutions $u$ and $\hat{u}$ are close (in the sense described above) on $\Omega$. At this point, we manage only to get a rough bound on $c^*(t)$, cf.~Theorem \ref{thm:claim boundary}, but we conjecture that a more sophisticated estimate is possible; see Section \ref{sec:conjecture}.

\section{Preliminaries}\label{sec:preliminaries}
We need a few fundamental results, before we can discuss the properties of solutions (Section \ref{sec:soln reg}) and the details of our results (Section \ref{sec:flux estimates} and further). We summarize these preliminaries in this section.
\begin{lemma}[Properties of the convolution, \cite{Folland} Propositions 8.8 and 8.9, p.~241]\label{lem:properties convolution}
Let $p,q\geq1$ be such that $1/p+1/q=1$. If $f\in L^p(\R^n)$ and $g\in L^q(\R^n)$, then
\begin{enumerate}
  \item $(f*g)(x)$ exists for all $x\in\R^n$;
  \item $f*g$ is bounded and uniformly continuous;
  \item $\|f*g\|_{L^\infty(\R^n)}\leq \|f\|_{L^p(\R^n)}\,\|g\|_{L^q(\R^n)}$.\label{lem:part inifinity bound, properties convolution}
\end{enumerate}
If moreover $p,q\in(1,\infty)$, then
\begin{enumerate}[resume]
  \item $f*g\in C_0(\R^n)$.
\end{enumerate}
Let $p,q,r\in[1,\infty]$ satisfy $1/p+1/q=1+1/r$. If $f\in L^p(\R^n)$ and $g\in L^q(\R^n)$, then
\begin{enumerate}[resume]
  \item $f*g\in L^r(\R^n)$;\label{lem:part in Lr, properties convolution}
  \item $\|f*g\|_{L^r(\R^n)}\leq \|f\|_{L^p(\R^n)}\,\|g\|_{L^q(\R^n)}$.\label{lem:part Lr bound, properties convolution}
\end{enumerate}
\end{lemma}
\begin{proof}
The proof can be found in \cite{Folland}, p.~241.
\end{proof}
Statement \ref{lem:part Lr bound, properties convolution} of Lemma \ref{lem:properties convolution} is called Young's inequality. It also holds for the convolution in time with upper bound $t$, which will appear in \eqref{eqn:solution convolution Green's function}. This is shown in the following corollary:
\begin{corollary}\label{cor:Lr bound conv time}
Let $T$ be fixed and let $p,q,r\in[1,\infty]$ satisfy $1/p+1/q=1+1/r$. If $f\in L^p([0,T])$ and $g\in L^q([0,T])$, then
\begin{enumerate}[resume]
  \item $f*_tg:=t\mapsto\Int{0}{t}{f(t-s)g(s)}{ds}\in L^r([0,T])$;
  \item $\|f*_tg\|_{L^r([0,T])}\leq \|f\|_{L^p([0,T])}\,\|g\|_{L^q([0,T])}$.\label{lem:part Lr bound, convolution time}
\end{enumerate}
\end{corollary}
\begin{proof}
The statement of this corollary follows from extension to $\R$ of $f$ and $g$ by zero outside $[0,T]$ and applying Lemma \ref{lem:properties convolution}, Parts \ref{lem:part in Lr, properties convolution} and \ref{lem:part Lr bound, properties convolution} (for $n=1$).
\end{proof}
The Green's function of the diffusion operator on $\R^n$ is (for general dimension $n$) given by
\begin{equation}\label{eqn:Green's function general dimension}
G_t(x) := (4\pi d t)^{-n/2} \mathrm{e}^{-|x|^2/4dt}.
\end{equation}
\begin{lemma}[Properties of the Green's function on $\R^2$]\label{lem:properties Green's function}
Consider the Green's function \eqref{eqn:Green's function general dimension} for dimension $n=2$.
\begin{enumerate}
  \item The gradient of the Green's function satisfies
    \begin{align}\label{eqn:lemma bound nabla G time}
\|\nabla G_\cdot(x)\|_{L^\infty(0,\infty)} := \sup_{\tau\in(0,\infty)}\|\nabla G_\tau(x)\|=
                \left\{
                   \begin{array}{ll}
                     0, & \hbox{$x=0$;} \\
                     \dfrac{8\mathrm{e}^{-2}}{\pi}\,|x|^{-3}, & \hbox{$x\in\R^2\setminus\{0\}$.}
                   \end{array}
                 \right.
    \end{align}
  \item For all $1\leq p\leq\infty$ there is a constant $c$ such that for all $t\in\Rp$
    \begin{equation}
    \|G_t(\cdot)\|_{L^p(\R^2)}\leq c\,t^{\frac1p-1}.
    \end{equation}
    The constant depends on $p$ and $d$.\label{lem:part Lp bound gradient, properties Green's function}
\end{enumerate}
\end{lemma}
\begin{proof}
\begin{enumerate}
  \item For all $x\in\R^2$ and all $\tau\in\Rp$
    \begin{equation}\label{eqn:Euclidean norm gradient G}
    \|\nabla G_\tau(x)\|=\dfrac{|x|}{8\pi d^2\tau^2}\,\mathrm{e}^{-|x|^2/4d\tau},
    \end{equation}
    where $\|\cdot\|$ denotes the Euclidean norm on $\R^2$. For $x=0$ we have that $\|\nabla G_\tau(0)\|=0$ for all $\tau\in(0,\infty)$, thus the corresponding part of \eqref{eqn:lemma bound nabla G time} follows.\\
    Next, we consider $x\neq0$. Note that for all such $x$
    \begin{align}
    \lim_{\tau\to 0}\|\nabla G_\tau(x)\|&=0,\\
    \lim_{\tau\to \infty}\|\nabla G_\tau(x)\|&=0.
    \end{align}
    Since the right-hand side in \eqref{eqn:Euclidean norm gradient G} is nonnegative and differentiable for all $\tau\in\Rp$, its maximum on $\Rp$ is attained where
    \begin{equation}
    \dfrac{\partial}{\partial\tau}\|\nabla G_\tau(x)\|=\dfrac{|x|}{4\pi d^2 \tau^3}\left(\dfrac{|x|^2}{8d\tau}-1\right)\,\mathrm{e}^{-|x|^2/4d\tau} =0,
    \end{equation}
    i.e.~at $\tau=|x|^2/8d$. Now the statement of the lemma follows:
    \begin{equation}
    \|\nabla G_\cdot(x)\|_{L^\infty(0,\infty)}=\|\nabla G_\tau(x)\|\Big|_{\tau=|x|^2/8d}= \dfrac{8\mathrm{e}^{-2}}{\pi}\,|x|^{-3}.
    \end{equation}
  \item The proof is a direct consequence of the statement in \cite{Hille_JEE} at the bottom of p.~432.
\end{enumerate}
\end{proof}

\section{Solution concepts and their regularity}\label{sec:soln reg}
For problem \eqref{eqn:system Full problem} we follow \cite{Denk0,Denk} by considering solutions in the sense of $L^2(\Omega)$-valued distributions on $[0,T]$. Our setting is a special case of the setting in \cite{Denk}. However, \cite{Denk} is one of the few works that we are aware of that consider maximal regularity issues for problems in {\em unbounded} domains. The seminal works by Solonnikov \cite{Solonnikov} and Lasiecka \cite{Lasieska} cover bounded domains $\Omega$ only.

We reformulate Theorem 2.1 in \cite{Denk} to obtain:
\begin{theorem}\label{thrm:denk et al}
If
\begin{itemize}
  \item $\phi\in H^1([0,T]; L^2(\Gamma))\cap L^2([0,T];H^1(\Gamma))$, and
  \item $u_0\in H^1(\Omega)$,
\end{itemize}
then Problem \eqref{eqn:system Full problem} has a unique solution
\begin{equation}
u \in H^1([0,T];L^2(\Omega))\cap L^2([0,T];H^2(\Omega)).
\end{equation}
\end{theorem}
\begin{proof}
The statement of this theorem is fully covered by Theorem 2.1 in \cite{Denk}. We now point out why we satisfy their conditions. Note that we use $p=2$ and $m=1$ in their setting. First, $\R$ is a so-called $\mathcal{HT}$-space, meaning that the Hilbert transform defines a bounded operator on $L^p(\R)$ for $1<p<\infty$ (cf. \cite{Riesz}, VII). The conditions (E), (LS), (SD) and (SB) from \cite{Denk} are easily verified for $\mathcal{A}u:=-d\Delta u$ and $\mathcal{B}u:=\nabla u\cdot n$. Regarding Condition (D) in \cite{Denk}, we note that in our case $f\equiv0$ and moreover, no compatibility condition (iv) is needed. In (iii), we use that $B^1_{2,2}(\Omega)=H^1(\Omega)$; see \cite{AdamsFournier} p.~231. A sufficient condition for (ii) to hold, is the one on $\phi$ given in the hypotheses of this theorem. We avoid -- in our setting unnecessary -- the use of fractional Sobolev spaces.
\end{proof}

Problem \eqref{eqn:system point source uhat} has a measure-valued right-hand side. \cite{Boccardo_ea} provide regularity results for weak solutions of non-linear parabolic problems with such measure-valued right-hand side. These apply to {\em bounded} domains with Dirichlet boundary condition and zero initial value.

We consider {\em mild solutions} to \eqref{eqn:system point source uhat} in the Banach space of finite Borel measures on $\R^2$, completed for the dual bounded Lipschitz norm $\|\cdot\|_{\mathrm{BL}}^*$ or Fortet-Mourier norm: $\overline{\mathcal{M}}(\R^2)_{\mathrm{BL}}$ (cf. \cite{Hille-Worm} and references found there). First, the diffusion semigroup $(S_t)_{t\geq0}$ on $\overline{\mathcal{M}}(\R^2)_{\mathrm{BL}}$ is defined for measures $\mu\in\mathcal{M}(\R^2)$ by convolution with the Green's function $G_t$ defined by \eqref{eqn:Green's function general dimension}, i.e.
\begin{equation}\label{eq:def diff on measures}
\pair{S_t\mu}{\varphi} := \pair{G_t*\mu}{\varphi}=\int_{\R^2}\int_{\R^2} G_t(x-y)\varphi(x)\,d\mu(y)\,dx
\end{equation}
for $\varphi\in C_b(\R^2)$. Thus, for positive $\mu$, $S_t\mu$ defines a positive linear functional on $C_c(\R^2)$, which is represented by a unique Radon measure according to the Riesz Representation Theorem. It is a finite measure because
\[
(S_t\mu)(\R^2)= \pair{S_t\mu}{\mathbf{1}}=\mu(\R^2)<\infty.
\]
Using the Jordan decomposition, we see that $S_t\mu\in\mathcal{M}(\R^2)$ for any $\mu\in\mathcal{M}(\R^2)$. One can check using \eqref{eq:def diff on measures} that $S_t$ is a bounded operator on $\mathcal{M}(\R^2)$ for $\|\cdot\|_{\mathrm{BL}}^*$. By continuity it extends to the completion $\overline{\mathcal{M}}(\R^2)_{\mathrm{BL}}$. Moreover, there exists $C>0$ such that
\[
\|S_t\nu\|_{\mathrm{BL}}^*\leq C\|\nu\|_{\mathrm{BL}}^*
\]
for all $t\geq 0$ and $\nu\in\overline{\mathcal{M}}(\R^2)_{\mathrm{BL}}$. Strong continuity of $(S_t)_{t\geq0}$ on $\overline{\mathcal{M}}(\R^2)_{\mathrm{BL}}$ can then be obtained from strong continuity on the dense subspace $\mathcal{M}(\R^2)$ that follows from \eqref{eq:def diff on measures} and \cite{Engel-Nagel}, Proposition I.5.3.

The mild solution to \eqref{eqn:system point source uhat} is now defined by
\begin{equation}\label{eqn:varOfConst}
\hat{\mu}(t) := S(t)\mu_0 + \Int{0}{t}{S(t-s)[\bar{\phi}(s)\delta_0]}{ds},
\end{equation}
for given initial measure $\mu_0\in\mathcal{M}(\R^2)$ (\cite{Pazy}, Ch.4, Def. 2.3, p.106). One can show that $\hat{\mu}\in C(\R_+,\overline{\mathcal{M}}(\R^2)_{\mathrm{BL}})$ whenever $\bar{\phi}\in L^1_{\mathrm{loc}}(\R_+)$.

If $\mu_0$ has density $\hat{u}_0$ with respect to Lebesgue measure $dx$ on $\R^2$, then according to \eqref{eq:def diff on measures} solution $\hat{\mu}(t)$ can be identified  with $\hat{u}(x,t)dx$ where the density function $\hat{u}$ is given by
\begin{align}
\nonumber \hat{u}(x,t) &= \Int{\R^2}{}{G_t(x-y)\hat{u}_0(y)}{dy} + \Int{0}{t}{G_{t-s}(x)\bar{\phi}(s)}{ds}\\
&=: (G_t*_x \hat{u}_0) (x) + (G_\cdot(x)*_t \bar{\phi}) (t). \label{eqn:solution convolution Green's function}
\end{align}
for all $(x,t)\in\R^2\times\Rp$. Here the notation $*_x$ and $*_t$ emphasizes that one takes convolution with respect to the space or time variable. Both have a regularising effect on the solution, that yields the following result for the restriction of $\hat{u}(t)$ to $\Omega$, the domain on which we compare with solution $u(t)$ to Problem \eqref{eqn:system Full problem}:

\begin{theorem}\label{thm:reg}
If $\hat{u}_0\in H^1(\R^2)$ and $\bar{\phi}\in H^1([0,T])$, then $\hat{u}$ (restricted to $\Omega$) satisfies
\begin{equation}
\hat{u} \in H^1([0,T];L^2(\Omega))\cap L^2([0,T];H^2(\Omega)).
\end{equation}
Moreover, $\partial_t\hat{u}(t)=d\Delta\hat{u}(t)$ in $L^2(\Omega)$ for almost every $t$ in $[0,T]$.
\end{theorem}
\begin{proof}
See Appendix.
\end{proof}

\section{Flux estimates}\label{sec:flux estimates}
In this section we present in Theorem \ref{thm:claim boundary} a bound on the difference between the fluxes on $\Gamma$ in \eqref{eqn:system Full problem} and \eqref{eqn:system point source uhat}. According to Theorem \ref{thrm:denk et al} and Theorem \ref{thm:reg}, under the conditions for which these results hold, $c^*(t)$ defined by \eqref{eq:def c-star} is finite for every $t\in[0,T]$. The difference between the solutions $u$ and $\hat{u}$ on $\Omega$ will be expressed in terms of $c^*(t)$, among others, in Section \ref{sec:exterior estimates}.

Throughout this section, we shall assume the conditions of Theorems \ref{thrm:denk et al} and \ref{thm:reg} on the data. Note that $\bar{\phi}\in H^1([0,T])$ implies that
\begin{equation}
\Int{0}{t}{\|\bar{\phi}\|^2_{L^1(0,\tau)}}{d\tau} \leq \mbox{$\frac{1}{2}$} t^2 \|\bar{\phi}\|^2_{L^2([0,T])}<\infty
\end{equation}
for all $0\leq t\leq T$.

Before getting at the main estimate for $c^*(t)$, we derive auxiliary results in Lemma \ref{lem:boundary flux uhat0=0} and Lemma \ref{lem:boundary flux}.
\begin{lemma}\label{lem:boundary flux uhat0=0}
Assume that $\hat{u}_0\equiv0$. Then, for all $t>0$ we have
\begin{equation}
\Int{0}{t}{\|d\nabla\hat{u}\cdot n\|^2_{L^2(\Gamma)}}{}\leq d^2 C_\Gamma \Int{0}{t}{\|\bar{\phi}\|^2_{L^1(0,\tau)}}{d\tau} < \infty,
\end{equation}
where
\[
C_\Gamma:= \int_\Gamma\|\nabla G_\cdot(x)\|^2_{L^\infty(0,\infty)}\,d\sigma>0
\]
is independent of $t$.
\end{lemma}
\begin{proof}
For $\hat{u}_0\equiv0$, the solution \eqref{eqn:solution convolution Green's function} of \eqref{eqn:system point source uhat} is given by
\begin{align}
\hat{u}(x,t)=\Int{0}{t}{G_{t-s}(x)\bar{\phi}(s)}{ds}.
\end{align}
Note that for $x\in\Gamma$ we have
\begin{align}\label{eqn:bound on flux uhat}
\nonumber \left| d\nabla\hat{u}(x,\tau)\cdot n(x) \right| &= \left| d\Int{0}{\tau}{\nabla G_{\tau-s}(x)\bar{\phi}(s)}{ds} \cdot n(x) \right|\\
\nonumber &\leq \left\| d\Int{0}{\tau}{\nabla G_{\tau-s}(x)\bar{\phi}(s)}{ds} \right\|\\
\nonumber &\leq d\,\|\nabla G_\cdot(x)\|_{L^\infty(0,\infty)} \, \Int{0}{\tau}{\left| \bar{\phi}(s) \right|}{ds}\\
&= d\,\|\nabla G_\cdot(x)\|_{L^\infty(0,\infty)} \, \| \bar{\phi} \|_{L^1(0,\tau)}.
\end{align}
We emphasize here that the infinity norm $\|\nabla G_\cdot(x)\|_{L^\infty(0,\infty)}$ denotes the supremum in the time domain for fixed $x$, cf.~\eqref{eqn:lemma bound nabla G time}. This observation leads to the following estimate
\begin{align}
\nonumber \Int{0}{t}{\|d\nabla\hat{u}(x,\tau)\cdot n(x)\|^2_{L^2(\Gamma)}}{d\tau} &= \Int{0}{t}{\Int{\Gamma}{}{|d\nabla\hat{u}(x,\tau)\cdot n(x)|^2}{d\sigma}}{d\tau}\\
&\leq d^2\,\Int{0}{t}{\Int{\Gamma}{}{\|\nabla G_\cdot(x)\|^2_{L^\infty(0,\infty)} \, \| \bar{\phi} \|^2_{L^1(0,\tau)}}{d\sigma}}{d\tau},
\end{align}
where \eqref{eqn:bound on flux uhat} is used in the second step. Thus, we have
\begin{align}\label{eqn:final bound on flux uhat=0}
\Int{0}{t}{\|d\nabla\hat{u}(x,\tau)\cdot n(x)\|^2_{L^2(\Gamma)}}{d\tau} \leq d^2\,\Int{0}{t}{\| \bar{\phi} \|^2_{L^1(0,\tau)}{d\tau}\,\Int{\Gamma}{}{\|\nabla G_\cdot(x)\|^2_{L^\infty(0,\infty)}}}{d\sigma}.
\end{align}
Since $\Gamma$ has finite length and it is the boundary of a set of which $0$ is an interior point, it follows from \eqref{eqn:lemma bound nabla G time} in Lemma \ref{lem:properties Green's function} that the second integral on the right-hand side of \eqref{eqn:final bound on flux uhat=0} is finite. This finishes the proof.
\end{proof}
In the next lemma we generalize this result to nonzero initial conditions.
\begin{lemma}\label{lem:boundary flux}
If $\hat{u}_0$ is such that $\nabla\hat{u}_0 \in L^p(\R^2)$ for some $2<p\leq\infty$, then
\begin{equation}
\Int{0}{t}{\|d\nabla\hat{u}\cdot n\|^2_{L^2(\Gamma)}}{} \leq d^2|\Gamma|C t^{\frac{2}{q}-1} \|\nabla \hat{u}_0\|^2_{L^p(\R^2)} + 2d^2 C_\Gamma \Int{0}{t}{\|\bar{\phi}\|^2_{L^1(0,\tau)}}{d\tau} <\infty,
\end{equation}
for all $t>0$, where $q:=p/(p-1)$, $C$ depends on $d$ and $q$ and $C_\Gamma$ is the constant from Lemma \ref{lem:boundary flux uhat0=0}.
\end{lemma}
\begin{proof}
In this case, the solution of \eqref{eqn:system point source uhat} is given by \eqref{eqn:solution convolution Green's function}. We start with the following estimate
\begin{align}
\nonumber \Int{0}{t}{\|d\nabla\hat{u}(x,\tau)\cdot n(x)\|^2_{L^2(\Gamma)}}{d\tau} &\leq 2\Int{0}{t}{\Int{\Gamma}{}{\left|d\nabla\Int{\R^2}{}{G_\tau(x-y)\hat{u}_0(y)}{dy}\cdot n(x)\right|^2}{d\sigma}}{d\tau}\\
&\hspace{-0.5cm}+ 2\Int{0}{t}{\Int{\Gamma}{}{\left|d\nabla\Int{0}{\tau}{G_{\tau-s}(x)\bar{\phi}(s)}{ds}\cdot n(x)\right|^2}{d\sigma}}{d\tau}.\label{eqn:estimate flux two parts}
\end{align}
The second term on the right-hand side is covered by Lemma \ref{lem:boundary flux uhat0=0}. Regarding the first term, we remark that, due to properties of the convolution,
\begin{equation}\label{eqn:convolution variable interchange}
\left|d\,\nabla\Int{\R^2}{}{G_\tau(x-y)\hat{u}_0(y)}{dy} \cdot n(x)  \right| = \left|d\,\Int{\R^2}{}{G_\tau(y)\nabla\hat{u}_0(x-y)}{dy} \cdot n(x)  \right|.
\end{equation}
We use Part \ref{lem:part inifinity bound, properties convolution} of Lemma \ref{lem:properties convolution} to estimate the right-hand side
\begin{align}
\nonumber\left|d\,\Int{\R^2}{}{G_\tau(y)\nabla\hat{u}_0(x-y)}{dy} \cdot n(x)  \right| &\leq d\,\left\| \Int{\R^2}{}{G_\tau(y)\nabla\hat{u}_0(\cdot-y)}{dy} \right\|_{L^\infty(\R^2)}\\
&\leq d\,\left\|\nabla\hat{u}_0 \right\|_{L^p(\R^2)} \, \left\|G_\tau\right\|_{L^q(\R^2)},\label{eqn:bound gradient initial condition}
\end{align}
with $q:=p/(p-1)$.\\
It follows from \eqref{eqn:convolution variable interchange}--\eqref{eqn:bound gradient initial condition} and Part \ref{lem:part Lp bound gradient, properties Green's function} of Lemma \ref{lem:properties Green's function} that
\begin{align}
\nonumber \Int{0}{t}{\Int{\Gamma}{}{\left|d\nabla\Int{\R^2}{}{G_\tau(x-y)\hat{u}_0(y)}{dy}\cdot n(x)\right|^2}{d\sigma}}{d\tau}&\\
\nonumber &\hspace{-1cm}\leq d^2\,\left\|\nabla\hat{u}_0 \right\|^2_{L^p(\R^2)}\, \Int{0}{t}{\Int{\Gamma}{}{\left\|G_\tau\right\|^2_{L^q(\R^2)}}{d\sigma}}{d\tau}\\
\nonumber &\hspace{-1cm}\leq c^2\,d^2\,|\Gamma|\,\left\|\nabla\hat{u}_0 \right\|^2_{L^p(\R^2)}\Int{0}{t}{\tau^{\frac2q-2}}{d\tau}\\
&\hspace{-1cm}=\frac{q\, c^2\,d^2\,|\Gamma|}{2-q}\,t^{\frac2q-1}\left\|\nabla\hat{u}_0 \right\|^2_{L^p(\R^2)},\label{eqn:estimate convolution uhat}
\end{align}
where $c$ depends on $q$ and $d$. We can perform the integration in time in the last step of \eqref{eqn:estimate convolution uhat} since the hypothesis $p>2$ implies $q<2$. The desired result follows by \eqref{eqn:estimate flux two parts} and the calculations in the proof of Lemma \ref{lem:boundary flux uhat0=0}:
\begin{align}
\nonumber \Int{0}{t}{\|d\nabla\hat{u}(x,\tau)\cdot n(x)\|^2_{L^2(\Gamma)}}{d\tau} \leq& \, \frac{2q \,c^2\,d^2\,|\Gamma|}{2-q}\,t^{\frac2q-1}\left\|\nabla\hat{u}_0 \right\|^2_{L^p(\R^2)}\\
&\hspace{-0.5cm}+  2d^2\,\Int{0}{t}{\| \bar{\phi} \|^2_{L^1(0,\tau)}{d\tau}\,\Int{\Gamma}{}{\|\nabla G_\cdot(x)\|^2_{L^\infty(0,\infty)}}}{d\sigma},
\end{align}
of which the right-hand side is finite for all finite $t$.
\end{proof}
\begin{remark}
A sufficient condition for $\nabla\hat{u}_0 \in L^p(\R^2)$ to hold, is $\hat{u}_0 \in W^{1,p}(\R^2)$. To this aim, one may start from $u_0 \in W^{1,p}(\Omega)$ to hold for the \textit{given} initial data. The remaining question is whether it is possible to find an extension $v_0$ on $\OO$ as in \eqref{eqn:extended initial condition} such that $\hat{u}_0 \in W^{1,p}(\R^2)$. This, however is guaranteed by Theorem 5.22 on p.~151 of \cite{AdamsFournier}.
\end{remark}
\begin{remark}
It is crucial that the gradient is applied to the initial condition in the computations starting at \eqref{eqn:convolution variable interchange} and further. Instead of \eqref{eqn:convolution variable interchange}--\eqref{eqn:bound gradient initial condition}, we could, along the same lines, have estimated
\begin{equation}
\left|d\,\nabla\Int{\R^2}{}{G_\tau(x-y)\hat{u}_0(y)}{dy} \cdot n(x)  \right| \leq d\,\left\|\hat{u}_0 \right\|_{L^p(\R^2)} \, \left\|\nabla G_\tau\right\|_{L^q(\R^2)},
\end{equation}
which requires only a condition on $\hat{u}_0$, not on its gradient, for the lemma. It follows from \cite{Hille_JEE} (p.~432, bottom) that for some constant $C$
\begin{equation}
\left\|\nabla G_\tau\right\|_{L^q(\R^2)}\leq C\,\tau^{\frac1q -\frac32}.
\end{equation}
This is a problem however, since similar arguments as in \eqref{eqn:estimate convolution uhat} would lead to
\begin{equation}
\Int{0}{t}{\left\|\nabla G_\tau\right\|^2_{L^q(\R^2)}}{d\tau}\leq C\Int{0}{t}{\tau^{\frac2q -3}}{d\tau},
\end{equation}
of which the right-hand side is not integrable for any $1\leq q\leq \infty$.
\end{remark}
We now come to the summarizing result of this section.
\begin{theorem}\label{thm:claim boundary}
Assume that the hypotheses of Theorems \ref{thrm:denk et al}  and \ref{thm:reg} and Lemma \ref{lem:boundary flux} hold. Then, for all $t>0$ the function $c^*$ defined by \eqref{eq:def c-star} satisfies
\begin{equation}\label{eq:estimate c-star}
c^*(t) \leq 2\int_0^t\|\phi\|^2_{L^2(\Gamma)} +
2d^2|\Gamma|C t^{\frac{2}{q}-1} \|\nabla \hat{u}_0\|^2_{L^p(\R^2)} +2C_\Gamma \int_0^t \|\bar{\phi}\|^2_{L^1(0,\tau)}\,d\tau.
\end{equation}

\end{theorem}
\begin{proof}
The statement of this theorem is a direct consequence of the observation
\begin{equation}\label{eqn:estimate flux difference in sum Minkovski}
\Int{0}{t}{\|\phi-d\nabla\hat{u}\cdot n\|^2_{L^2(\Gamma)}}{} \leq 2\Int{0}{t}{\|\phi\|^2_{L^2(\Gamma)}}{} + 2\Int{0}{t}{\|d\nabla\hat{u}\cdot n\|^2_{L^2(\Gamma)}}{}.
\end{equation}
The first term is finite due to the assumption that $\phi\in L^2([0,T];L^2(\Gamma))$ for all $T\in\Rp$ (see Section \ref{sec:two problems}). The second term was estimated in Lemma \ref{lem:boundary flux}.
\end{proof}
\begin{remark}
Estimate \eqref{eq:estimate c-star} is unsatisfactory for $t$ close to zero. However, it shows for large $t$ that on the long run the difference between the fluxes on $\Gamma$ is dominated by the prescribed fluxes $\phi$ at $\Gamma$ and $\bar{\phi}$ at the point source at 0, rather than the initial condition, which is clear intuitively. In Section \ref{sec:conjecture} we provide a further discussion of the behaviour of $c^*(t)$.
\end{remark}

\section{Estimates in the exterior -- Proof of Theorem \ref{thm:estimate u exterior}}\label{sec:exterior estimates}

We can now proof our main result, an estimate for the difference between the solutions $u$ of \eqref{eqn:system Full problem} and $\hat{u}$ of \eqref{eqn:system point source uhat} (using the solution concept explained in Section \ref{sec:soln reg}):

\begin{proof} {\it (Theorem \ref{thm:estimate u exterior}).}\
Let $\psi\in C^\infty_c(\overline{\Omega})$ and $h\in C^\infty_c([0,T])$ be test functions. Put $(\psi\otimes h)(x,t):=\psi(x)h(t)$. Then according to Theorem \ref{thrm:denk et al} and Theorem \ref{thm:reg} one has
\begin{align}
\nonumber \pair{\partial_tu -\partial_t\hat{u}}{\psi\otimes h} &= d\pair{\Delta u - \Delta \hat{u}}{\psi\otimes h}\\
&= \int_0^T\left\{\Int{\Gamma}{}{\left(\phi(t)-d\nabla\hat{u}(t)\cdot n\right)\psi}{}\right\}h(t)dt \label{eq:computation norm} \\
& \qquad - d\int_0^T\left\{\Int{\Omega}{}{\left(\nabla u - \nabla \hat{u}\right)\cdot\nabla\psi}{}\right\}h(t)dt.\nonumber
\end{align}
Because of the regularity of the solutions $u$ and $\hat{u}$ identity \eqref{eq:computation norm} extends to functions $f\in L^1([0,T],H^1(\Omega))$ by continuity:
\begin{align}
\nonumber \pair{\partial_tu -\partial_t\hat{u}}{f} & = \Int{0}{T}{\Int{\Gamma}{}{\left(\phi(t)-d\nabla\hat{u}(t)\cdot n\right)f(x,t)}{d\sigma(x)}}{dt}\\
&\qquad - d\Int{0}{T}{\Int{\Omega}{}{\left(\nabla u - \nabla \hat{u}\right)\cdot\nabla f(x,t)}{dx}}{dt}.\label{eq:aux towards norm}
\end{align}
Now take $f(x,t):=(u(x,t)-\hat{u}(x,t))h(t)$ with $h\in C^\infty_c([0,T])$ arbitrary. Then the regularity of $u$ and $\hat{u}$ and \eqref{eq:aux towards norm} imply that
\begin{equation}
\dfrac12 \dfrac{d}{dt}\|u-\hat{u}\|^2_{L^2(\Omega)} + d \|\nabla u-\nabla \hat{u}\|^2_{L^2(\Omega)} = \Int{\Gamma}{}{(u-\hat{u})(\phi-d\nabla\hat{u}\cdot n)}{}.
\end{equation}
Add $d\|u-\hat{u}\|^2_{L^2(\Omega)}$ to both sides and integrate in time from $0$ to arbitrary $t$:
\begin{equation}
\dfrac12 \|u-\hat{u}\|^2_{L^2(\Omega)} + d \Int{0}{t}{\|u-\hat{u}\|^2_{H^1(\Omega)}}{} = \Int{0}{t}{\Int{\Gamma}{}{(u-\hat{u})(\phi-d\nabla\hat{u}\cdot n)}{}}{} + d\Int{0}{t}{\|u-\hat{u}\|^2_{L^2(\Omega)}}{},\label{eqn:integrated}
\end{equation}
where we have used that $u$ and $\hat{u}$ are initially equal on $\Omega$. Apply the Cauchy-Schwarz inequality and use the result of Theorem \ref{thm:claim boundary} to obtain
\begin{align}
\nonumber\Int{0}{t}{\Int{\Gamma}{}{(u-\hat{u})(\phi-d\nabla\hat{u}\cdot n)}{}}{} &\leq \left(\Int{0}{t}{\|u-\hat{u}\|^2_{L^2(\Gamma)}}{}\right)^{\frac12}\,\left(\Int{0}{t}{\|\phi-d\nabla\hat{u}\cdot n\|^2_{L^2(\Gamma)}}{}\right)^{\frac12}\\
&= \sqrt{c^*(t)}\,\left(\Int{0}{t}{\|u-\hat{u}\|^2_{L^2(\Gamma)}}{}\right)^{\frac12}.\label{eqn:innerproduct Cauchy-Schwarz}
\end{align}
Since $H^1(\Omega)\hookrightarrow L^2(\Gamma)$, according to the Boundary Trace Imbedding Theorem (cf.~\cite{AdamsFournier}, Theorem 5.36, p.~164) there is a constant $\bar{c}=\bar{c}(\Omega)>0$ such that
\begin{equation}
\|u-\hat{u}\|_{L^2(\Gamma)}\leq \bar{c} \, \|u-\hat{u}\|_{H^1(\Omega)},
\end{equation}
which can be used to further estimate \eqref{eqn:innerproduct Cauchy-Schwarz}:
\begin{align}
\Int{0}{t}{\Int{\Gamma}{}{(u-\hat{u})(\phi-d\nabla\hat{u}\cdot n)}{}}{} &\leq \sqrt{c^*(t)}\,\bar{c}\,\left(\Int{0}{t}{\|u-\hat{u}\|^2_{H^1(\Omega)}}{}\right)^{\frac12}.
\end{align}
For arbitrary $\ep>0$, Young's inequality yields the following estimate on the right-hand side:
\begin{align}
\sqrt{c^*(t)}\,\bar{c}\,\left(\Int{0}{t}{\|u-\hat{u}\|^2_{H^1(\Omega)}}{}\right)^{\frac12} &\leq \frac{1}{2\ep}c^*(t)\bar{c}^2 + \frac{\ep}{2}\Int{0}{t}{\|u-\hat{u}\|^2_{H^1(\Omega)}}{}.\label{eqn:Young}
\end{align}
Take $\ep\in(0,2d)$. Then \eqref{eqn:integrated}--\eqref{eqn:Young} together yield
\begin{equation}
\|u-\hat{u}\|^2_{L^2(\Omega)} + (2d-\ep) \Int{0}{t}{\|u-\hat{u}\|^2_{H^1(\Omega)}}{} \leq \frac{1}{\ep}c^*(t)\bar{c}^2 + 2d\Int{0}{t}{\|u-\hat{u}\|^2_{L^2(\Omega)}}{},\label{eqn:ineq before Gronwall}
\end{equation}
or
\begin{equation}
\|u-\hat{u}\|^2_{L^2(\Omega)} + \underbrace{(2d-\ep) \Int{0}{t}{\|\nabla u-\nabla\hat{u}\|^2_{L^2(\Omega)}}{}}_{\geq0} \leq \frac{1}{\ep}c^*(t)\bar{c}^2 + \ep\Int{0}{t}{\|u-\hat{u}\|^2_{L^2(\Omega)}}{}.\label{eqn:final ineq before Gronwall}
\end{equation}
It follows that
\begin{align}
\|u-\hat{u}\|^2_{L^2(\Omega)} &\leq \frac{1}{\ep}c^*(t)\bar{c}^2 + \ep\Int{0}{t}{\|u-\hat{u}\|^2_{L^2(\Omega)}}{},
\end{align}
and due to a version of Gronwall's lemma\footnote{A specific form of Theorem 1 on p.~356 of \cite{Mitrinovic}.}
\begin{align}
\|u-\hat{u}\|^2_{L^2(\Omega)} &\leq \frac{1}{\ep}c^*(t)\bar{c}^2\,\mathrm{e}^{\ep t},\label{eqn:bound L2}
\end{align}
where we use that $c^*(\cdot)$ is (by definition) non-decreasing. Note that $\ep$ is arbitrary but fixed, thus $1/\ep<\infty$. We obtain \eqref{eqn:thm bound L2} by defining $c_1:=\bar{c}^2/\ep$.\\
From \eqref{eqn:ineq before Gronwall} it also follows that
\begin{align}
\Int{0}{t}{\|u-\hat{u}\|^2_{H^1(\Omega)}}{} &\leq \frac{1}{\ep(2d-\ep)}c^*(t)\bar{c}^2 + \frac{2d}{2d-\ep}\Int{0}{t}{\|u-\hat{u}\|^2_{L^2(\Omega)}}{}.
\end{align}
The upper bound \eqref{eqn:bound L2} now implies
\begin{align}
\nonumber \Int{0}{t}{\|u-\hat{u}\|^2_{H^1(\Omega)}}{} &\leq \frac{1}{\ep(2d-\ep)}c^*(t)\bar{c}^2 + \frac{2d}{\ep^2(2d-\ep)}c^*(t)\bar{c}^2\,(\textrm{e}^{\ep t}-1)\\
&\leq \frac{2d}{\ep^2(2d-\ep)}c^*(t)\bar{c}^2\,\textrm{e}^{\ep t},\label{eqn:estimate integrated H1 norm final}
\end{align}
where we use that $\ep<2d$ in the second step. The second statement of the theorem now follows by defining $c_2:=2d\bar{c}^2/(\ep^2(2d-\ep))$.
\end{proof}
\begin{remark}
In principle, \eqref{eqn:estimate integrated H1 norm final} can be optimized in $\ep$ for every $t$ separately, to get an optimal $\ep=\ep(t)$. After substitution of this $\ep(t)$, \eqref{eqn:thm bound H1 integrated} becomes independent of $\ep$. However, its $t$-dependence obviously becomes more complicated. Further details on this aspect are omitted here.
\end{remark}
\begin{remark}
The fact that the estimates in Theorem \ref{thm:estimate u exterior} are linear in $c^*$ relates nicely to our Conjecture \ref{conj:c* small}; see Section \ref{sec:conjecture} below. If indeed $c^*$ is small or even goes to zero, then the same holds for $\|u(\cdot,t)-\hat{u}(\cdot,t)\|^2_{L^2(\Omega)}$ and $\Int{0}{t}{\|u-\hat{u}\|^2_{H^1(\Omega)}}{}$.
\end{remark}


\section{Conjecture}
\label{sec:conjecture}
The estimate \eqref{eqn:estimate flux difference in sum Minkovski} is a very crude way to find an upper bound on $c^*(t)$. In the following (deliberately vague) conjecture, we express under which conditions we expect $c^*(t)$ to be smaller than the upper bound of Theorem \ref{thm:claim boundary} suggests.

\begin{conjecture}\label{conj:c* small}
The upper bound $c^*$ can be much smaller than Theorem \ref{thm:claim boundary} suggests. Ideally it goes to zero.
\end{conjecture}
Conjecture \ref{conj:c* small} is based on the following considerations:
\begin{itemize}
  \item Once the geometry and $\phi$ on $\Gamma$ are given, there still is a lot of freedom in dealing with the reduced problem \eqref{eqn:system point source uhat}. We can choose $\bar{\phi}$ and $v_0$. Our conjecture is that a smart choice of $\bar{\phi}$ and $v_0$ can produce a flux on $\Gamma$ that mimics well $\phi$ and gives more than merely a \textit{bounded} difference.
  \item Initially, during a small time interval, the initial condition should induce a sufficiently close flux. To this aim an appropriate $v_0$ has to be provided.
  \item At a certain moment, mass originating from the source starts reaching the boundary. From then onwards, the mimicking flux should be -- with some delay -- mainly due to $\bar{\phi}$.
  \item Let $|\OO|$ denote a typical length scale of the object $\Obj$ (e.g.~its diameter). The quantity $|\OO|^2/d$ is a typical timescale for points to travel the distance from source to boundary. This is also the timescale at which the transition between the above two bullet points takes place.
  \item The shape of object $\Obj$ is important. An intuitive guess is that a small object $\Obj$ can be better approximated. As the point source emits mass at the same rate in all directions, we expect a better approximation also to be possible if $\Gamma$ is radially symmetric with respect to the origin, and $\phi$ is constant on $\Gamma$ (in space, not necessarily in time). A generalization of the latter condition would be to have $\phi$ defined on a more general $\Gamma$, but to have an extension to a ball $B(0,R)$ such that $\Gamma\subset B(0,R)\subset\R^2$, and this extension is radially symmetric around the origin on $B(0,R)$.
\end{itemize}
The above statement was written under the assumption that in general the (normal component of the) flux is directed outward on $\Gamma$. For a mass sink, \textit{mutatis mutandis} the same considerations hold.

\section*{Acknowledgements}
This work was started after fruitful discussions during the workshop ``Modelling with Measures: from Structured Populations to Crowd Dynamics", organized at the Lorentz Center in Leiden, The Netherlands.\\
We thank Jan de Graaf, Georg Prokert, Patrick van Meurs, Upanshu Sharma (Eindhoven), Ulrich R\"ude (Erlangen) and Thomas Seidman (Baltimore) for sharing their thoughts with us.

\begin{appendix}
\section*{Appendix -- Proof of Theorem \ref{thm:reg}}\label{app:proof reg}
\begin{proof}
Note that for $\hat{u}_0\in H^1(\R^2)$, the function $\hat{u}_1:=G*_x\hat{u}_0$ is a solution of
\begin{align}
\left\{
  \begin{array}{ll}
    \dfrac{\partial u}{\partial t}= d \Delta u, & \hbox{on $\R^2\times\Rp$;} \\
    u(0) = \hat{u}_0, & \hbox{on $\R^2$,}
  \end{array}
\right.
\end{align}
which is unique and satisfies
\begin{equation}
\hat{u}_1 \in H^1([0,T];L^2(\R^2))\cap L^2([0,T];H^2(\R^2))
\end{equation}
due to \cite{Denk}, Theorem 2.1, where the domain is taken to be $\R^2$.\\
\\
Define $\hat{u}_2:= G*_t\bar{\phi}$. Then $\hat{u}_2$ satisfies
\begin{align}
\nonumber \|\hat{u}_2\|_{L^2(\Omega)} &= \left(\Int{\Omega}{}{\left|\Int{0}{t}{G_{t-s}(x)\,\bar{\phi}(s)}{ds}\right|^2}{dx}  \right)^{1/2}\\
\nonumber &\leq \Int{0}{t}{\left(\Int{\Omega}{}{\left| G_{t-s}(x)\,\bar{\phi}(s) \right|^2}{dx} \right)^{1/2}}{ds}\\
\nonumber &\leq \Int{0}{t}{\|G_{t-s}\|_{L^2(\R^2)}|\bar{\phi}(s)|}{ds}\\
&\leq \Int{0}{t}{c\,(t-s)^{-1/2}|\bar{\phi}(s)|}{ds}.\label{eqn:estimate uhat2}
\end{align}
In the second step we used Minkowski's inequality for integrals (see \cite{Stein}, p.~271), whereas the last inequality follows from Part \ref{lem:part Lp bound gradient, properties Green's function} of Lemma \ref{lem:properties Green's function}. Since $t\mapsto c\,t^{-1/2}\in L^1([0,T])$ and by assumption $\bar{\phi}\in L^2([0,T])$, Corollary \ref{cor:Lr bound conv time} applied to \eqref{eqn:estimate uhat2} yields
\begin{equation}\label{eqn:G*barPhi in L2}
\|\hat{u}_2\|_{L^2(\Omega)}\in L^2([0,T]).
\end{equation}
Because $G_\cdot(x)$ and $\partial_tG_\cdot(x)$ are in $L^1_{\mathrm{loc}}(\R_+)$ for $x\neq 0$ and $\partial_t\bar{\phi}\in L^2(\R_+)$, one has in the sense of distributions
\begin{equation}\label{eq:derivative convolution}
\partial_t\bigl( G_\cdot(x)*\bar{\phi} \bigr) = \bigl( \partial_t G_\cdot(x)\bigr)*\bar{\phi} = G_\cdot(x) * \bigl(\partial_t \bar{\phi}\bigr).
\end{equation}
Thus we can repeat the argument leading to \eqref{eqn:G*barPhi in L2}, replacing $\bar{\phi}$ by $\partial_t\bar{\phi}$, and obtain
\begin{equation}\label{eqn:d/dt G*barPhi in L2}
\|\partial_t\hat{u}_2\|_{L^2(\Omega)}\in L^2([0,T]).
\end{equation}
We conclude from \eqref{eqn:G*barPhi in L2} and \eqref{eqn:d/dt G*barPhi in L2} that
\begin{equation}\label{eqn:hat u2 in H1 time L2 space}
\hat{u}_2 \in H^1([0,T];L^2(\Omega)).
\end{equation}
It follows from \eqref{eqn:Green's function general dimension}, with $n=2$ that
\begin{align}
\partial_{x_i} G_t(x)&=\dfrac{x_i}{8\pi d^2t^2}\mathrm{e}^{-|x|^2/4dt},\,\text{and}\\
\partial_{x_i}\partial_{x_j}G_t(x) &= \dfrac{1}{8\pi d^2t^2}\mathrm{e}^{-|x|^2/4dt}\left[\delta_{ij}-\dfrac{x_ix_j}{2dt} \right],
\end{align}
where $\delta_{ij}$ denotes the Kronecker delta. The gradient is bounded in the following way:
\begin{align}
\nonumber |\nabla G_t(x)|^2 &\leq \sup_{t>0}|\nabla G_t(x)|^2\\
\nonumber &= \sup_{t>0}  \dfrac{|x|^2}{64\pi^2 d^4t^4}\mathrm{e}^{-|x|^2/2dt}\\
&= \dfrac{1}{|x|^6}\sup_{u>0}\dfrac{u^4}{4\pi^2}\mathrm{e}^{-u},
\end{align}
for all $t>0$ and for all $x\in\Omega$, where we substituted $u:=|x|^2/2dt$ to obtain the constant $c_1:=\sup_{u>0}\dfrac{u^4}{4\pi^2}\mathrm{e}^{-u}$, which is independent of $|x|$, $t$, $d$. Thus
\begin{equation}
|\nabla G_t(x)|^2 \leq \dfrac{c_1}{|x|^6}.\label{eqn:bound nabla G L2Omega}
\end{equation}
For a matrix $M\in\R^{n\times n}$, as matrix norm we use the Frobenius norm and denote it by $\|\cdot\|_F$:
\begin{equation}
\|M\|_F:= \sqrt{\sum_{i,j}|M_{ij}|^2}.
\end{equation}
In a similar way as for $\nabla G$, we estimate the Hessian matrix
\begin{align}
\nonumber \|D^2G_t(x)\|_F^2 &\leq \sup_{t>0}\left(\sum_{i=1}^2\sum_{j=1}^2 \dfrac{1}{64\pi^2 d^4t^4}\mathrm{e}^{-|x|^2/2dt}\left[\delta_{ij}-\dfrac{x_ix_j}{2dt} \right]^2  \right)\\
\nonumber &= \sup_{t>0}  \dfrac{1}{64\pi^2 d^4t^4}\mathrm{e}^{-|x|^2/2dt}\left(2-\dfrac{|x|^2}{dt}+\dfrac{|x|^4}{4d^2t^2}\right)\\
\nonumber &= \dfrac{1}{|x|^8}\sup_{u>0}\dfrac{u^4}{4\pi^2}\mathrm{e}^{-u}\left(2-2u+u^2\right)\\
&= \dfrac{c_2}{|x|^8},\label{eqn:bound D2 G L2Omega}
\end{align}
for all $t>0$ and for all $x\in\Omega$. Now we show that $\partial_{x_i}G_t$ and $\partial_{x_i}\partial_{x_j}G_t$ are in $L^2(\Omega)$, both with uniform upper bound in $t$:
\begin{align}
\nonumber\|\partial_{x_i}G_t\|^2_{L^2(\Omega)} &= \Int{\Omega}{}{|\partial_{x_i}G_t(x)|^2}{dx}\\
\nonumber &\leq \Int{\Omega}{}{|\nabla G_t(x)|^2}{dx}\\
&\stackrel{\eqref{eqn:bound nabla G L2Omega}}{\leq} \Int{\Omega}{}{\dfrac{c_1}{|x|^6}}{dx}=:C_1<\infty,\label{eqn:bound partial G L2 Omega}
\end{align}
where we use that $0$ is an interior point of $\OO=\R^2\setminus\overline{\Omega}$. Also
\begin{align}
\nonumber\|\partial_{x_i}\partial_{x_j}G_t\|^2_{L^2(\Omega)} &= \Int{\Omega}{}{|\partial_{x_i}\partial_{x_j}G_t(x)|^2}{dx}\\
\nonumber &\leq \Int{\Omega}{}{\|D^2 G_t(x)\|_F^2}{dx}\\
&\stackrel{\eqref{eqn:bound D2 G L2Omega}}{\leq} \Int{\Omega}{}{\dfrac{c_2}{|x|^8}}{dx}=:C_2<\infty.\label{eqn:bound second partial G L2 Omega}
\end{align}
For brevity, we now use the index notation for derivatives and, for $|\alpha|\in\{1,2\}$. Like in \eqref{eqn:estimate uhat2}, using Minkowski's integral inequality, we obtain that
\begin{equation}\label{eqn:bound dAlpha Minkowski}
\|\partial_x^\alpha \hat{u}_2\|_{L^2(\Omega)} \leq \Int{0}{t}{\|\partial_x^\alpha G_{t-s}\|_{L^2(\Omega)}|\bar{\phi}(s)|}{ds}.
\end{equation}
Due to \eqref{eqn:bound partial G L2 Omega}--\eqref{eqn:bound second partial G L2 Omega}, for each $|\alpha|\in\{1,2\}$ and for each $\tau>0$
\begin{equation}
\|\partial_x^\alpha G_{\tau}\|_{L^2(\Omega)}\in L^\infty([0,T])\subset L^1([0,T]).
\end{equation}
Hence, the fact that $\bar{\phi}\in L^2([0,T])$ yields via Part \ref{lem:part Lr bound, convolution time} of Corollary \ref{cor:Lr bound conv time} that
\begin{equation}\label{eqn:dAlpha in L2 0T}
\Int{0}{t}{\|\partial_x^\alpha G_{t-s}\|_{L^2(\Omega)}|\bar{\phi}(s)|}{ds} \in L^2([0,T]),
\end{equation}
for each $|\alpha|\in\{1,2\}$. It follows from \eqref{eqn:G*barPhi in L2}, \eqref{eqn:bound dAlpha Minkowski} and \eqref{eqn:dAlpha in L2 0T} that
\begin{equation}
\hat{u}_2\in L^2([0,T];H^2(\Omega)).
\end{equation}
Together with \eqref{eqn:hat u2 in H1 time L2 space}, this finishes the proof of the first part.

The last statement follows from \eqref{eq:derivative convolution} and a similar result for the spatial derivatives. For all $\psi\in C^\infty_c(\overline{\Omega})$ and $h\in C^\infty_c(\R_+)$, $\psi\otimes h(x,t):=\psi(x)h(t)$ is in $C^\infty_c(\overline{\Omega}\times\R_+)$ and one has
\begin{align*}
\pair{\partial_t \hat{u}}{\psi\otimes h} &= \pair{\bigl(\partial_t G_\cdot) *_x \hat{u}_0}{\psi\otimes h} + \int_{\overline{\Omega}} \pair{\partial_t\bigl[ G_\cdot(x)\bigr] *_t\bar{\phi}}{h} \psi(x)\,dx \\
& = \pair{d(\Delta G_\cdot) *_x \hat{u}_0}{\psi\otimes h} + \int_{\overline{\Omega}} \pair{d[\Delta G_\cdot(x)] *_t \bar{\phi}}{h} \psi(x)\,dx \\
& = \pair{d\Delta (G_\cdot *_x \hat{u}_0)}{\psi\otimes h} + \pair{d\Delta(G_\cdot *_t \bar{\phi})}{\psi\otimes h}\\
&= \pair{d\Delta \hat{u}}{\psi\otimes h}.
\end{align*}
By density of $C_c^\infty(\overline{\Omega})\otimes C^\infty_c(\R_+)$ in the space of test functions $\mathcal{D}(\overline{\Omega}\times\R_+)$ we obtain $\partial_t \hat{u}=d\Delta\hat{u}$ in the sense of distributions on $\overline{\Omega}\times\R_+$. Since both are given by (locally integrable) functions according to the first part of the proof, $\partial_t\hat{u}(t)=d\Delta \hat{u}(t)$ for almost every $t$.
\end{proof}
\begin{remark}
The estimates \eqref{eqn:bound partial G L2 Omega}--\eqref{eqn:bound second partial G L2 Omega} hinge on the fact that $\Omega$ is bounded away from $0$, where the integrand is singular.
\end{remark}
\end{appendix}


\begin{thebibliography}{99}
\bibitem{AdamsFournier}
     \newblock R.A.~Adams and J.J.F.~Fournier,
     \newblock ``Sobolev Spaces,"
     \newblock 2nd edition,  Academic Press, 2003.

\bibitem{Baluska}
    \newblock F.~Balu\v{s}ka, J.~\v{S}amaj and D.~Menzel,
    \newblock \emph{Polar transport of auxin: carrier-mediated flux across the plasma membrane or neurotransmitter-like secretion?},
    \newblock Trends in Cell Biology, \textbf{13} (2003), 282--285.

\bibitem{vBerkel}
    \newblock K.~van Berkel, R.J.~de Boer, B.~Scheres, and K.~ten Tusscher,
    \newblock \emph{Polar auxin transport: models and mechanisms},
    \newblock Development, \textbf{140} (2013), 2253--2268.
		
\bibitem{Boccardo_ea}
		\newblock L.~Boccardo, A.~Dall'Aglio, Th.~Gallou\"et, and L.~Orsina,
    \newblock \emph{Nonlinear parabolic equations with measure data},
    \newblock J. Funct. Anal., \textbf{147} (1997), 237--258.

\bibitem{Boot}
    \newblock K.J.M.~Boot, K.R.~Libbenga, S.C.~Hille, R.~Offringa and B.~van Duijn,
    \newblock \emph{Polar auxin transport: an early invention},
    \newblock Journal of Experimental Botany, \textbf{63} (2012), 4213--4218.

\bibitem{LeBris}
    \newblock E.~Canc\`{e}s and C.~Le Bris,
    \newblock \emph{Mathematical modeling of point defects in materials science},
    \newblock Mathematical Models and Methods in Applied Sciences, \textbf{23} (2013), 1795--1859.

\bibitem{Chavarria}
    \newblock A.~Chavarr\'{i}a-Krauser and M.~Ptashnyk,
    \newblock \emph{Homogenization of long-range auxin transport in plant tissues},
    \newblock Nonlinear Analysis: Real World Applications, \textbf{11} (2010), 4524--4532.

\bibitem{Denk0}
    \newblock R.~Denk, M.~Hieber and J.~Pr\"{u}ss,
    \newblock $\mathcal{R}$-boundedness, Fourier multipliers and problems of elliptic and parabolic type,
    \newblock Mem. Am. Math. Soc, \textbf{788} (2003).

\bibitem{Denk}
    \newblock R.~Denk, M.~Hieber and J.~Pr\"{u}ss,
    \newblock \emph{Optimal $L^p$-$L^q$-estimates for parabolic boundary value problems with inhomogeneous data},
    \newblock Math. Z., \textbf{257} (2007), 193--224.

\bibitem{Engel-Nagel}
		\newblock K.-J.~Engel and R.~Nagel,
		\newblock ``One Parameter Semigroups for Linear Evolution Equations",
		\newblock Springer-Verlag, New York, 2000.

\bibitem{Folland}
     \newblock G.B.~Folland,
     \newblock ``Real Analysis: Modern Techniques and Their Applications,"
     \newblock 2nd edition,  Wiley, New York, 1999.

\bibitem{Griffiths}
     \newblock D.J.~Griffiths,
     \newblock ``Introduction to Electrodynamics,"
     \newblock 3rd edition,  Pearson Education, 2008.

\bibitem{Gulikers_Jstat}
    \newblock L.~Gulikers, J.H.M.~Evers, A.~Muntean and A.V.~Lyulin,
    \newblock \emph{The effect of perception anisotropy on particle systems describing pedestrian flows in corridors},
    \newblock Journal of Statistical Mechanics: Theory and Experiment, \textbf{P04025} (2013).

\bibitem{Hille_JEE}
    \newblock S.C.~Hille,
    \newblock \emph{Local well-posedness of kinetic chemotaxis models},
    \newblock Journal of Evolution Equations, \textbf{8} (2008), 423--448.
		
\bibitem{Hille-Worm}
		\newblock S.C.~Hille, and D.T.H.~Worm,
    \newblock \emph{ Embedding of semigroups of Lipschitz maps into positive linear semigroups on ordered Banach spaces generated by measures},
    \newblock Integr. Equ. Oper. Theory, \textbf{63} (2009), 351--371.
		
\bibitem{Hirokawa_ea}
		\newblock N.~Hirokawa, S.~Niwa, Y.~Tanaka,
		\newblock \emph{Molecular motors in neurons: Transport mechanisms and roles in brain function, development, and disease},
		\newblock Neuron, \textbf{68} (2010), 610--638.

\bibitem{Jackson}
     \newblock J.D.~Jackson,
     \newblock ``Classical Electrodynamics,"
     \newblock 3rd edition,  John Wiley and Sons, 1999.
		
\bibitem{Jaeder-Nagel}
		\newblock H.M.~J\"ager and S.R.~Nagel,
     \newblock \emph{Physics of the granular state},
     \newblock Science \textbf{255} (1982), 1523-1531.
		
\bibitem{Kramer}
    \newblock E.M.~Kramer,
    \newblock \emph{Computer models of auxin transport: a review and commentary},
    \newblock Journal of Experimental Botany, \textbf{59} (2008), 45--53.

\bibitem{Lasieska}
    \newblock I.~Lasiecka,
    \newblock \emph{Unified theory for abstract boundary problems--a semigroup approach},
    \newblock Appl. Math. Optim. \textbf{6} (1980), 287--333.


\bibitem{Lions}
     \newblock J.D.~Lions, E.~Magenes,
     \newblock ``Non-Homogeneous Boundary Value Problems and Applications,"
     \newblock Springer Verlag, 1972.

\bibitem{Lin-Edwards:1997}
		\newblock Y.~Liu, R.H.~Edwards,
		\newblock \emph{The role of vesicular transport proteins in synaptic transmission and neural degeneration},
		\newblock Ann. Rev. Neurosci., \textbf{20} (1997), 125--156.

\bibitem{Merks}
    \newblock R.M.H.~Merks, Y.~Van de Peer, D.~Inz\'{e} and G.T.S.~Beemster,
    \newblock \emph{Canalization without flux sensors: a traveling-wave hypothesis},
    \newblock Trends in Plant Science, \textbf{12} (2007), 384--390.

\bibitem{vMeurs}
    \newblock P.~van Meurs, A.~Muntean and M.A.~Peletier,
    \newblock \emph{Upscaling of dislocation walls in finite domains}, preprint,
    \newblock Eur. J. Appl. Math, \emph{to appear}.

\bibitem{Mitrinovic}
     \newblock D.S.~Mitrinovi\'{c}, J.E.~Pe\v{c}ari\'{c} and A.M.~Fink,
     \newblock ``Inequalities Involving Functions and Their Integrals and Derivatives,"
     \newblock Kluwer Academic Publishers, Dordrecht, 1991.

\bibitem{Pazy}
    \newblock A.~Pazy,
    \newblock ``Semigroups of Linear Operators and Applications to Partial Differential Equations",
    \newblock Springer-Verlang, New York, 1983.	
		

\bibitem{Raven}
    \newblock J.A.~Raven,
    \newblock \emph{Polar auxin transport in relation to long-distance transport of nutrients in the Charales},
    \newblock Journal of Experimental Botany, \textbf{64} (2013), 1--9.

\bibitem{Riesz}
    \newblock M.~Riesz,
    \newblock \emph{Sur les fonction conjugu\'ees},
    \newblock Math. Zeit. \textbf{27}(1) (1928), 218--244.
		
 \bibitem{Rude}
    \newblock U.~R{\"u}de, H.~K{\"o}stler, M.~Mohr,
  \newblock \emph{Accurate multigrid techniques for computing singular solutions of elliptic problems},
  \newblock Eleventh Copper Mountain Conference on Multigrid Methods, 2003.

\bibitem{Seidman}
    \newblock T.I.~Seidman, M.K.~Gobbert, D.W.~Trott and M.~Kru\u{z}\'{\i}k,
    \newblock \emph{Finite element approximation for time-dependent diffusion with measure-valued source},
    \newblock Numer.~Math., \textbf{122} (2012), 709--723.
		
\bibitem{Solonnikov}
    \newblock V.A.~Solonnikov, ,
    \newblock \emph{On boundary value problems for linear parabolic systems of differential equations of general form},
    \newblock Trudy Mat. Fust. Steklov \textbf{83} (1965), 3--163 (Russian).
	\newblock Engl. Transl.: Proc. Steklov Inst. Math. \textbf{83} (1965), 1--184.

\bibitem{Stein}
    \newblock E.M.~Stein,
    \newblock ``Singular Integrals and Differentiability Properties of Functions",
    \newblock Princeton University Press, Princeton, New Jersey, 1970.			

				
\end{thebibliography}
\end{document}